\documentclass[10pt,reqno]{amsart}

\usepackage[textwidth=5in,textheight=8in,centering]{geometry}
\usepackage[T1]{fontenc}
\usepackage{lmodern}
\usepackage{microtype}
\usepackage{flafter}
\usepackage{amsmath,amssymb}
\usepackage{xcolor}
\usepackage{tikz}
\usetikzlibrary{arrows.meta,decorations.pathreplacing}
\usepackage{hyperref}
\definecolor{brickblue}{RGB}{46,88,132}
\definecolor{brickorange}{RGB}{180,111,55}
\tikzset{every picture/.style={line width=.55pt}}
\hypersetup{colorlinks=true,linkcolor=brickblue,citecolor=brickblue,
 urlcolor=brickblue,pdftitle={Fourfold Is Enough: Connected LEGO Models from One Brick Type},
 pdfauthor={Joshua S. Gans}}
\newtheorem{theorem}{Theorem}
\newtheorem{lemma}[theorem]{Lemma}
\title[Fourfold Is Enough]{Fourfold Is Enough:\\
Connected LEGO Models from One Brick Type}
\author{Joshua S. Gans}
\address{Rotman School of Management, University of Toronto, Toronto, Canada}
\date{}

\begin{document}
\begin{abstract}
Can a LEGO model built from many kinds of rectangular bricks be rebuilt
using only the familiar $2\times4$ brick? Enlarging every dimension by
four makes the pieces fit, but they must also connect. We give a
four-layer arrangement that joins every finite face-connected union of
lattice cubes, using eight upright bricks per source cube, and show that
no smaller universal enlargement works. A second question has an
unexpected answer: if connectivity is set aside, the minimum integer
scale is always $1$, $2$, or $4$, never $3$. The missing value follows
from a checkerboard argument applied separately inside each enlarged cell.
\end{abstract}
\maketitle

\section{One brick, any shape?}

A LEGO model built from rectangular bricks draws on a collection of pieces: small bricks
for narrow details, long bricks for walls, and larger pieces to tie
sections together. Suppose we replace this collection with a supply of
just the familiar $2\times4$ brick. Can we reproduce the same shape?
Small details may force us to build a larger version, but could one
fixed enlargement work for every model, however intricate? The question
asks how much of a construction's variety comes from its inventory,
and how much can be recovered by rearranging one simple piece.

There are two obstacles. The bricks must fill the prescribed shape
exactly, including its recesses and holes. They must also connect:
LEGO bricks fasten above and below, while two bricks that merely meet
along their sides do not attach. Filling small parts of a model
independently can therefore leave a collection of separate assemblies.
Bridging those parts requires bricks to cross their boundaries without
protruding into the empty space around them. A construction must arrange
these bridges throughout an arbitrary three-dimensional shape.

Fourfold enlargement provides just enough room for a simple solution.
Represent the shape by unit cubes. Each becomes a $4\times4\times4$
block, which holds two upright $2\times4\times1$ bricks in each of four
layers (Figure~\ref{fig:fourfold}). To connect the result, give the layers
different jobs. Turn the bricks between the first two layers so that
the pieces within each enlarged cube attach to one another. Use the
third layer to bridge neighbouring cubes along rows, and the fourth to
bridge them along columns. Cubes above one another connect through
their common horizontal face. The task is to make these bridges cover
each layer exactly; the staggered arrangement in
Figure~\ref{fig:connections} does so.

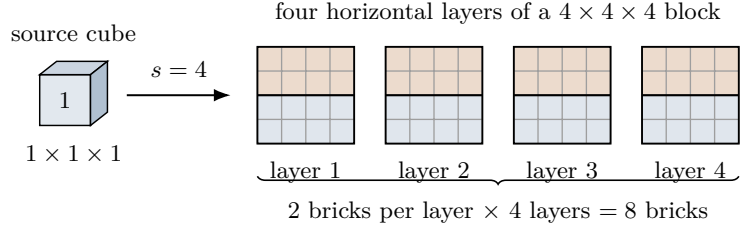
\begin{figure}[htbp]
\centering
\begin{tikzpicture}[x=.32cm,y=.32cm,>=Latex,
 every node/.style={font=\fontsize{9}{11}\selectfont}]
 \filldraw[fill=brickblue!14] (0,.75) rectangle (2.10,2.85);
 \filldraw[fill=brickblue!25]
  (0,2.85)--(.70,3.40)--(2.80,3.40)--(2.10,2.85)--cycle;
 \filldraw[fill=brickblue!38]
  (2.10,.75)--(2.80,1.30)--(2.80,3.40)--(2.10,2.85)--cycle;
 \node at (1.05,1.80) {$1$};
 \node[above=4pt] at (1.40,3.40) {source cube};
 \node[below=4pt] at (1.40,.75) {$1\times1\times1$};

 \draw[->,line width=.75pt]
  (3.60,2.00)--(7.90,2.00)
  node[midway,above=3pt] {$s=4$};
 \node[above=5pt] at (18.95,4.00)
  {four horizontal layers of a $4\times4\times4$ block};

 \foreach \xx/\layer in {9/1,14.3/2,19.6/3,24.9/4}{
  \begin{scope}[shift={(\xx,0)}]
   \fill[brickblue!15] (0,0) rectangle (4,2);
   \fill[brickorange!25] (0,2) rectangle (4,4);
   \draw[line width=.5pt,black!38,step=1] (0,0) grid (4,4);
   \draw[line width=.7pt] (0,0) rectangle (4,4);
   \draw[line width=.8pt] (0,2)--(4,2);
   \node[below=3pt] at (2,0) {layer $\layer$};
  \end{scope}
 }
 \draw[decorate,decoration={brace,mirror,amplitude=3pt},line width=.6pt]
  (9,-1.35)--(28.9,-1.35)
  node[midway,below=6pt]
  {$2$ bricks per layer $\times$ $4$ layers $=8$ bricks};
\end{tikzpicture}
\caption{Filling the enlarged shape takes eight bricks per source cube:
two in each of four layers. Repeating this arrangement independently need
not connect adjacent cubes.}
\label{fig:fourfold}
\end{figure}

This brings together two familiar mathematical questions: filling boxes
with a fixed brick, studied by de~Bruijn~\cite{DeBruijn1969}, and
connected LEGO structures, studied by Eilers~\cite{Eilers2016}.
Contact graphs also appear in computational methods for constructing
brick sculptures~\cite{TestuzSchwartzburgPauly2013}. Here the region is
prescribed and the brick type is fixed. After proving the connected
construction, we set connectivity aside and ask when a smaller
enlargement can be tiled. The minimum integer scale is then always
one, two, or four: three is never needed.

\section{Four layers that connect everything}

We model a shape $M$ as a finite, nonempty union of unit cubes with
vertices in $\mathbb Z^3$. It is \emph{face-connected} if any two cubes
can be joined by a chain of cubes sharing faces. Write $V$ for the
number of cubes. The replacement bricks are integer translates of
$[0,4]\times[0,2]\times[0,1]$ or
$[0,2]\times[0,4]\times[0,1]$: they may turn horizontally but remain
upright. Coordinates use stud spacing horizontally and brick height
vertically; enlargement multiplies all three coordinates by the same factor.

Two replacement bricks have a \emph{contact} when they lie in consecutive
height layers and their horizontal footprints overlap in at least one
unit square. A construction is connected when any two bricks can be
joined by a chain of contacts. This idealisation records attachments;
it does not assert mechanical stability. We reproduce the occupied
region and may change the boundaries between its original pieces.

\begin{theorem}\label{thm:fourfold}
Every face-connected model $M$ has an exact fourfold enlargement made
from upright $2\times4\times1$ bricks with a connected contact graph.
The construction uses $8V$ bricks. No smaller positive real enlargement
factor works for every such model.
\end{theorem}

The arrangement in Figure~\ref{fig:connections} uses the same four
layers over every source layer. Its ingredient is a way to bridge all
the boundaries in a row of enlarged cells at once. For a row of $k$
cells, the required footprint is a $4k\times4$ strip. Put a
$2\times4$ brick footprint at each end. The remaining $4(k-1)\times4$
rectangle takes two rows of $4\times2$ footprints, starting two units
from the left edge. Their long sides span the intervals
\[
 [4j-2,4j+2],\qquad j=1,\ldots,k-1.
\]
These intervals fill the space between the end pieces, and each crosses
the source-cell boundary at $4j$. When $k=1$, the two end pieces alone
fill the square. Turning this pattern through a right angle gives the
column construction.

\begin{figure}[htbp]
\centering
\begin{tikzpicture}[x=.32cm,y=.32cm,>=Latex,
  every node/.style={font=\fontsize{9}{11}\selectfont},
  brick/.style={draw=black,line width=.65pt},
  seam/.style={black,line width=.8pt,densely dashed},
  stage/.style={draw=black,line width=.65pt,rounded corners=1.4pt,
                minimum height=6.2mm,align=center,inner xsep=2.8pt}]

  \begin{scope}[shift={(0,8)}]
    \draw[line width=.5pt,black!35,step=1] (0,0) grid (4,4);
    \filldraw[brick,fill=brickblue!15] (0,0) rectangle (4,2);
    \filldraw[brick,fill=brickorange!30] (0,2) rectangle (4,4);
    \node[above=3pt] at (2,4) {layer 1};
    \node[below=3pt] at (2,0) {east--west};
  \end{scope}
  \begin{scope}[shift={(5.5,8)}]
    \draw[line width=.5pt,black!35,step=1] (0,0) grid (4,4);
    \filldraw[brick,fill=brickblue!15] (0,0) rectangle (2,4);
    \filldraw[brick,fill=brickorange!30] (2,0) rectangle (4,4);
    \node[above=3pt] at (2,4) {layer 2};
    \node[below=3pt] at (2,0) {north--south};
  \end{scope}

  \begin{scope}[shift={(11,8)}]
    \draw[line width=.5pt,black!35,step=1] (0,0) grid (12,4);
    \filldraw[brick,fill=brickblue!30] (0,0) rectangle (2,4);
    \filldraw[brick,fill=brickblue!30] (10,0) rectangle (12,4);
    \foreach \xx in {2,6}{
      \filldraw[brick,fill=brickorange!15] (\xx,0) rectangle ++(4,2);
      \filldraw[brick,fill=brickorange!30] (\xx,2) rectangle ++(4,2);
    }
    \draw[seam] (4,-.30)--(4,4.30);
    \draw[seam] (8,-.30)--(8,4.30);
    \node[above=3pt] at (6,4.30) {layer 3: a row of three cells};
    \node[below=4pt] at (6,-.30) {both seams crossed};
  \end{scope}

  \begin{scope}[shift={(25,0)}]
    \draw[line width=.5pt,black!35,step=1] (0,0) grid (4,12);
    \filldraw[brick,fill=brickblue!30] (0,0) rectangle (4,2);
    \filldraw[brick,fill=brickblue!30] (0,10) rectangle (4,12);
    \foreach \yy in {2,6}{
      \filldraw[brick,fill=brickorange!15] (0,\yy) rectangle ++(2,4);
      \filldraw[brick,fill=brickorange!30] (2,\yy) rectangle ++(2,4);
    }
    \draw[seam] (-.30,4)--(4.30,4);
    \draw[seam] (-.30,8)--(4.30,8);
    \node[above=3pt] at (2,12) {layer 4};
  \end{scope}

  \node[stage,fill=brickorange!15,minimum width=12mm] (h) at (2,3.35)
    {layer 1};
  \node[stage,fill=black!15,minimum width=12mm] (v) at (7,3.35)
    {layer 2};
  \node[stage,fill=brickorange!15,minimum width=17mm] (hr) at (13,3.35)
    {layer 3};
  \node[stage,fill=black!15,minimum width=16mm] (vr) at (20,3.35)
    {layer 4};
  \draw[->,line width=.7pt] (h)--(v);
  \draw[->,line width=.7pt] (v)--(hr);
  \draw[->,line width=.7pt] (hr)--(vr);
  \node[align=center] at (11,1.20)
    {four target-height layers per source layer};
\end{tikzpicture}
\caption{Four layers, viewed from above. The first two turn the bricks
within each enlarged square. The third bridges every boundary along a
row; the fourth uses the rotated construction along columns. Dashed lines
mark boundaries between enlarged source cells. The arrows give the order
from bottom to top.}
\label{fig:connections}
\end{figure}
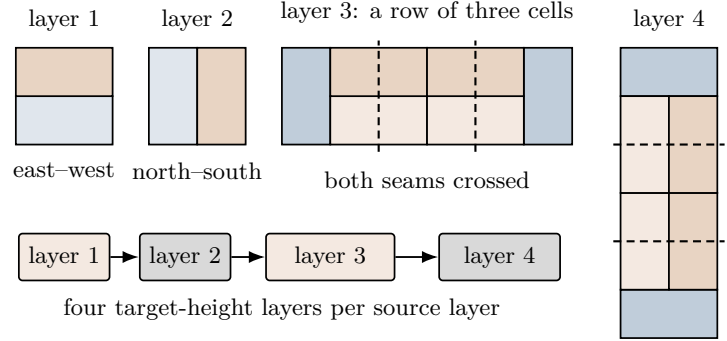

\begin{proof}
Each source cube becomes a $4\times4\times4$ block. Work through the
four brick layers over each horizontal source layer, as follows.

\emph{First connect each enlarged cell.} In the bottom layer of each
block, place two east--west bricks. In the second layer, place two
north--south bricks. Each upper brick overlaps both lower bricks in a
$2\times2$ square. These four bricks are therefore connected.

\emph{Next connect neighbouring cells in the same source layer.}
In the third brick layer, apply the strip pattern to every maximal
consecutive row of occupied source cells. These strips partition the
occupied footprint, even when it has holes or separate parts. A brick
crosses every boundary between east--west neighbours. Its portion in
each enlarged cell overlaps a second-layer brick there, so it joins the
connected groups on both sides of the boundary. Every third-layer brick
is connected to the group of each enlarged cell in which its footprint
occupies positive area.

In the fourth layer, use the turned strip pattern on maximal columns.
Any portion of a fourth-layer brick inside an enlarged cell overlaps
a third-layer brick there, which is already connected to that cell's
group. The fourth layer therefore joins every pair of north--south
neighbours. All overlaps used here have integer coordinates and positive
area, so they contain a unit square and give contacts. This connects
all bricks over each face-connected part of the source layer.

\emph{Finally connect successive source layers.} Two source cubes
sharing a horizontal face have the same enlarged square footprint.
Choose a unit square in that footprint. The top brick layer of the
lower block and the bottom brick layer of the upper block each contain
a brick covering this square. Those bricks have a contact. We have
therefore connected across every face adjacency of the source; since
the source is face-connected, the whole replacement is connected.

The enlarged region has volume $64V$, and each brick has volume $8$,
so the construction uses $8V$ bricks. For sharpness, take a single
source cube. Enlarging it by any real factor less than four leaves both
horizontal sides shorter than the long side of a target brick.
It cannot contain even one such brick.
\end{proof}

The construction also gives a simple placement order when every occupied
source column extends without gaps down to ground level. Within a source layer,
all four brick layers cover the same footprint; between source layers,
the upper footprint is contained in the lower one. Thus every brick
above ground overlaps a brick immediately below it. Placing bricks in
increasing order of height gives every nonground brick an already placed
lower contact. This condition concerns support by contact, without
claiming stability or connectedness at every intermediate stage.

\section{Why the minimum scale is never three}

Four is a universal answer, but some models need less enlargement.
For this section, ask only whether the shape can be tiled, without
requiring the resulting bricks to connect. Restrict the enlargement
factor to a positive integer, so the source lattice maps into the brick
lattice. At scale one, a $4\times2\times1$ block already is a brick;
at scale two, a $2\times1\times1$ block becomes two stacked bricks;
a single cube needs scale four (Figure~\ref{fig:scales}). Surprisingly,
these are all the possible minimum integer scales.

Because every brick has height one, we may examine horizontal layers
separately. Write $P$ for the union of unit squares in a source layer.
At integer scale $s$, that layer produces $s$ identical brick layers,
each with footprint $sP$. Consequently, the scaled model can be tiled
exactly when each of its scaled source footprints can be tiled by
$2\times4$ rectangles, allowing both orientations.

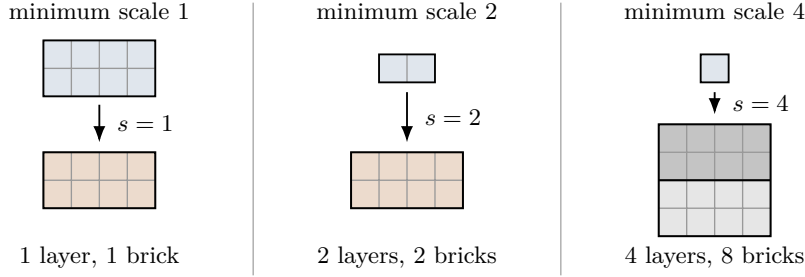
\begin{figure}[htbp]
\centering
\begin{tikzpicture}[x=.37cm,y=.37cm,>=Latex,
 every node/.style={font=\fontsize{9}{11}\selectfont}]
 \draw[black!35,line width=.5pt] (9.5,-1.40)--(9.5,8.35);
 \draw[black!35,line width=.5pt] (20.5,-1.40)--(20.5,8.35);

 \node at (4,8.05) {minimum scale $1$};
 \fill[brickblue!15] (2,5) rectangle (6,7);
 \draw[line width=.5pt,black!40,step=1] (2,5) grid (6,7);
 \draw[line width=.7pt] (2,5) rectangle (6,7);
 \draw[->,line width=.7pt]
  (4,4.70)--(4,3.35) node[midway,right=3pt] {$s=1$};
 \fill[brickorange!25] (2,1) rectangle (6,3);
 \draw[line width=.5pt,black!40,step=1] (2,1) grid (6,3);
 \draw[line width=.75pt] (2,1) rectangle (6,3);
 \node at (4,-.75) {$1$ layer, $1$ brick};

 \node at (15,8.05) {minimum scale $2$};
 \fill[brickblue!15] (14,5.5) rectangle (16,6.5);
 \draw[line width=.5pt,black!40] (15,5.5)--(15,6.5);
 \draw[line width=.7pt] (14,5.5) rectangle (16,6.5);
 \draw[->,line width=.7pt]
  (15,5.15)--(15,3.35) node[midway,right=3pt] {$s=2$};
 \fill[brickorange!25] (13,1) rectangle (17,3);
 \draw[line width=.5pt,black!40,step=1] (13,1) grid (17,3);
 \draw[line width=.75pt] (13,1) rectangle (17,3);
 \node at (15,-.75) {$2$ layers, $2$ bricks};

 \node at (26,8.05) {minimum scale $4$};
 \filldraw[fill=brickblue!15,line width=.7pt]
  (25.5,5.5) rectangle (26.5,6.5);
 \draw[->,line width=.7pt]
  (26,5.15)--(26,4.35) node[midway,right=3pt] {$s=4$};
 \fill[black!10] (24,0) rectangle (28,2);
 \fill[black!23] (24,2) rectangle (28,4);
 \draw[line width=.5pt,black!40,step=1] (24,0) grid (28,4);
 \draw[line width=.75pt] (24,0) rectangle (28,4);
 \draw[line width=.8pt] (24,2)--(28,2);
 \node at (26,-.75) {$4$ layers, $8$ bricks};
\end{tikzpicture}
\caption{All three minimum geometric scales occur. The top row shows
source footprints, each one cube high; the bottom shows a single brick
layer after enlargement. The counts include all height layers.}
\label{fig:scales}
\end{figure}

Doubling has a useful interpretation: $2P$ can be tiled by $2\times4$
rectangles if and only if $P$ can be tiled by dominoes. One direction
comes from doubling a domino tiling. For the reverse, every rectangle
in a tiling of $2P$ must align with the grid of enlarged $2\times2$
source squares. To see this, choose the lowest occupied unit square
in the leftmost occupied column. The rectangle covering it must have
its lower-left corner there, at a grid vertex. Both side lengths are
even, so this rectangle is a union of whole $2\times2$ grid squares.
Remove it and repeat. All rectangles align, and halving their coordinates
gives a domino tiling of $P$. In graph language, domino tilings are
perfect matchings of the source-cell adjacency graph: each domino pairs
two neighbouring cells, and each cell belongs to one pair.

The missing scale is explained by a stronger implication. A global
checkerboard count can rule out some tilings; applying it within each
enlarged cell forces the shape of a tiling at a coarser scale.

\begin{lemma}\label{lem:three}
If $3P$ can be tiled by $2\times4$ rectangles, then $P$ can be tiled
by $2\times2$ squares, and hence by dominoes.
\end{lemma}

\begin{proof}
Give the fine unit square with lower-left corner $(x,y)$ checkerboard
weight $(-1)^{x+y}$. Each enlarged source square is a $3\times3$ block
with total weight either $+1$ or $-1$; call this its sign. We count the
contributions of the tiling rectangles to this block separately.

The intersection of a rectangle with a block has weight zero whenever
one of its side lengths is even. A tile's short side of length two
therefore contributes only when it crosses a block boundary as $1+1$.
Its long side of length four splits at block boundaries as $3+1$,
$2+2$, or $1+3$; only the first and last possibilities contribute.
Thus a contributing tile meets exactly four blocks in a $2\times2$
arrangement, cutting out a $1\times1$ or $1\times3$ strip, or its
rotation, in each (Figure~\ref{fig:checkerboard}).

Each of these odd strips has the same weight as the block it meets.
Indeed, its lower-left corner has local coordinates zero or two in
both directions within the block, so its first square has the block's
sign. An odd alternating row has that same total sign, as does an odd
column. Every other tile contributes zero to every block it meets.
If $n$ contributing tiles meet a given block, additivity of weight gives
\[
 \text{sign of the block}=n\,(\text{sign of the block}),
\]
so $n=1$. Every occupied block therefore belongs to exactly one group
of four determined by a contributing tile. All four blocks in each
group are occupied, since the tile lies inside $3P$. Shrinking these
groups back to the source grid partitions $P$ into $2\times2$ squares.
Each such square splits into two dominoes.
\end{proof}

\begin{figure}[htbp]
\centering
\begin{tikzpicture}[x=.53cm,y=.53cm,
 every node/.style={font=\fontsize{9}{11}\selectfont}]
 \foreach \x in {0,...,5}{
  \foreach \y in {0,...,5}{
   \pgfmathtruncatemacro{\parity}{mod(\x+\y,2)}
   \ifnum\parity=0
    \fill[black!12] (\x,\y) rectangle ++(1,1);
   \fi
  }
 }
 \fill[brickorange!35,opacity=.7] (2,2) rectangle (6,4);
 \draw[step=1,black!35,line width=.5pt] (0,0) grid (6,6);
 \draw[line width=.8pt] (0,0) rectangle (6,6);
 \draw[line width=.8pt] (3,0)--(3,6);
 \draw[line width=.8pt] (0,3)--(6,3);
 \draw[brickorange!80!black,line width=1.1pt] (2,2) rectangle (6,4);
 \node at (1.35,1.3) {$+1$};
 \node at (4.5,1.3) {$-1$};
 \node at (1.35,4.7) {$-1$};
 \node at (4.5,4.7) {$+1$};
 \node[above=4pt] at (3,6) {four $3\times3$ blocks};
 \draw[->,line width=.7pt] (7,3)--(9.1,3);
 \begin{scope}[shift={(10,2)}]
  \fill[brickblue!15] (0,0) rectangle (2,2);
  \draw[step=1,line width=.6pt] (0,0) grid (2,2);
  \draw[line width=.8pt] (0,0) rectangle (2,2);
  \node[below=4pt,align=center] at (1,0) {one group of\\four source cells};
 \end{scope}
\end{tikzpicture}
\caption{A contributing rectangle cuts off an odd strip in each of four
blocks. Each strip has the same checkerboard imbalance as its block
(the signs shown). Exactly one such rectangle meets each occupied block,
so the corresponding source cells partition into $2\times2$ squares.}
\label{fig:checkerboard}
\end{figure}
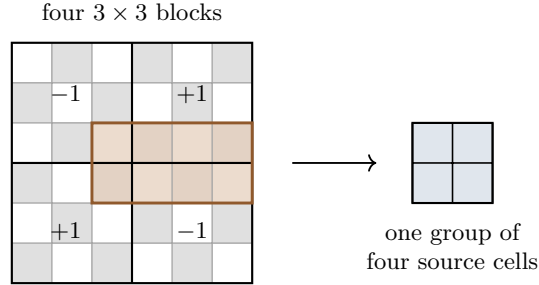

\begin{theorem}\label{thm:scales}
The minimum positive integer scale for a geometric tiling of any model
is $1$, $2$, or $4$. Scale $2$ works exactly when every horizontal
source footprint is domino-tileable.
\end{theorem}

\begin{proof}
The doubling argument characterises scale two. If scale three works,
Lemma~\ref{lem:three} makes every source footprint domino-tileable,
so scale two already works. Scale four always works: the eight-brick
filling of each enlarged cube needs no connectivity assumption.
The examples in Figure~\ref{fig:scales} realise each remaining value.
\end{proof}

Checkerboard arguments are a standard tool for tilings~\cite{Golomb1994}.
Here their local use does more than obstruct a tiling: it partitions the
source into larger squares. The connected construction uses a different
principle, assigning separate layers to separate directions of adjacency.
Together they leave a natural question: which models admit a
\emph{connected} reconstruction at scale one or two? The domino test
settles whether the pieces can fit at scale two, but the contacts between
successive layers still have to join them.

\smallskip
\noindent\textit{AI assistance.}
OpenAI Codex (GPT-5.6 Sol and GPT-6) assisted with proof exploration,
computational checks, literature searches, and drafting and revision.
Refine.ink provided feedback on an earlier draft. The author is
responsible for the mathematical arguments, their verification, and
the final text.

\smallskip
\noindent\textit{Trademark notice.}
LEGO is a trademark of the LEGO Group, which does not endorse this work.


\begin{thebibliography}{4}
\bibitem{DeBruijn1969}
N.~G. de~Bruijn, Filling boxes with bricks,
\emph{Amer. Math. Monthly} \textbf{76} (1969), 37--40.

\bibitem{Eilers2016}
S.~Eilers, The LEGO counting problem,
\emph{Amer. Math. Monthly} \textbf{123} (2016), 415--426.

\bibitem{Golomb1994}
S.~W. Golomb, \emph{Polyominoes: Puzzles, Patterns, Problems, and Packings},
2nd ed., Princeton University Press, Princeton, NJ, 1994.

\bibitem{TestuzSchwartzburgPauly2013}
R.~Testuz, Y.~Schwartzburg, and M.~Pauly,
Automatic generation of constructable brick sculptures,
in \emph{Eurographics 2013: Short Papers}, 2013, 81--84.
\end{thebibliography}
\end{document}